\documentclass[10pt]{amsart}
\usepackage{epsfig,amsfonts,amsthm,amssymb,latexsym,amsmath}

\newcommand{\N}{{\mathbb N}}
\newcommand{\Z}{{\mathbb Z}}

\theoremstyle{plain}
\newtheorem{theorem}{Theorem}[section]
\newtheorem{corollary}[theorem]{Corollary}

\theoremstyle{remark}

\theoremstyle{definition}
\newtheorem{definition}[theorem]{Definition}

\title{On $k$-Fibonacci sums by matrix methods}
\vspace{5pt}

\author{\scriptsize Gamaliel Cerda}
\date{}

\begin{document}
\maketitle

\vspace{-20pt}

\medskip
\noindent

\begin{abstract}
In this paper, some $k$-Fibonacci and $k$-Lucas with arithmetic indexes sums are derived by using
the matrices $R_{a}=\left[ \begin{array}{lr} L_{k,a}  & -(-1)^{a} \\ 1 &  0  \end{array}\right]$ and $S_{a}=\frac{1}{2}\left[ \begin{array}{lr} L_{k,a} & \Delta_{a} \\ 1 &  L_{k,a}  \end{array}\right]$, where $\Delta_{a}=L_{k,a}^{2}-4(-1)^{a}$.

The most notable side of this paper is our proof method, since all the identities used in the proofs of main theorems are proved previously by using the matrices 
$R_{a}$ and $S_{a}$, with $a\in \N$. Although the identities we proved are known, our proofs are not encountered in the $k$-Fibonacci and $k$-Lucas numbers literature.
\end{abstract}

\medskip
\noindent

\noindent
\medskip

\section{Introduction}\label{sec:1}
One of the more studied sequences is the Fibonacci sequence \cite{1}, and it has been generalized in many ways \cite{2,3}. Here, we use the following one-parameter generalization of the Fibonacci sequence.
\begin{definition}\label{def:11}
For any integer number $k\geq1$, the $k$-th Fibonacci sequence, say $\{F_{k,n}\}_{n\in \N}$ is defined recurrently by
\begin{equation}
F_{k,n+1}=kF_{k,n}+F_{k,n-1},\ n\geq1,
\end{equation}
where $F_{k,0}=0$ and $F_{k,1}=1$.
\end{definition}

Note that for $k=1$ the classical Fibonacci sequence is obtained while for $k=2$ we obtain the Pell sequence. Some of the properties that the $k$-Fibonacci numbers verify and that we will need later are summarized below \cite{4}:

[Binet's formula] $F_{k,n}=\frac{\sigma_{1}^{n}-\sigma_{2}^{n}}{\sigma_{1}-\sigma_{2}}$, where $\sigma_{1}=\frac{k+\sqrt{k^{2}+4}}{2}$ and $\sigma_{2}=\frac{k-\sqrt{k^{2}+4}}{2}$. These roots verify $\sigma_{1}+\sigma_{2}=k$ and $\sigma_{1}\sigma_{2}=-1$.

This paper presents an interesting investigation about some special relations between matrices and $k$-Fibonacci and $k$-Lucas numbers. This investigation is valuable, since it provides students to use their theoretical knowledge to obtain new $k$-Fibonacci and $k$-Lucas identities with arithmetic indexes by different methods. 

We focus here on the subsequences of $k$-Fibonacci numbers with indexes in an arithmetic sequence, say $an+r$ for fixed integers $a$, $r$ with $0\leq r \leq a-1$. Several formulas for the sums of such numbers are deduced by matrix methods.

\section{Main theorems}
Let us denote $F_{k,n+1}+F_{k,n-1}$ by $L_{k,n}$ (the $k$-Lucas numbers).

\begin{theorem}
Let $a$ be a fixed positive integer. If $T$ is a square matrix with $T^{2}=L_{k,a}T-(-1)^{a}I$ and $I$ the matrix identity of order 2. Then, 
\begin{equation}\label{e1}
T^{n}=\frac{1}{F_{k,a}}\left(F_{k,an}T-(-1)^{a}F_{k,a(n-1)}I\right),
\end{equation}
for all $n\in \Z$.
\end{theorem}

\begin{proof}
If $n=0$, the proof is obvious because $F_{k,-a}=-(-1)^{a}F_{k,a}$. It can be shown by induction that $F_{k,a}T^{n}=F_{k,an}T-(-1)^{a}F_{k,a(n-1)}I$, for every positive integer $n$. We now show that 
\begin{equation}
T^{-n}=\frac{1}{F_{k,a}}\left(F_{k,a(-n)}T-(-1)^{a}F_{k,a(-n-1)}I\right).
\end{equation}

Let $U=L_{k,a}I-T=(-1)^{a}T^{-1}$, then
\begin{align*}
U^{2}=(L_{k,a}I-T)^{2}&=L_{k,a}^{2}I-2L_{k,a}T+T^{2}\\
&=L_{k,a}(L_{k,a}I-T)-(-1)^{a}I=L_{k,a}U-(-1)^{a}I,
\end{align*}
this shows that $U^{n}=\frac{1}{F_{k,a}}\left(F_{k,an}U-(-1)^{a}F_{k,a(n-1)}I\right)$.\\

That is, $F_{k,a}((-1)^{a}T^{-1})^{n}=F_{k,an}(L_{k,a}I-T)-(-1)^{a}F_{a(n-1)}I$. Therefore 
\begin{align*}
(-1)^{an}(F_{k,a}T^{-n})&=-F_{k,an}T+(L_{k,a}F_{k,an}-(-1)^{a}F_{k,a(n-1)})I\\
&=-F_{k,an}T+F_{k,a(n+1)}I.
\end{align*}

Thus, 
\begin{equation}\label{e2}
T^{-n}=\frac{1}{F_{k,a}}\left(-(-1)^{-an}F_{k,an}T+(-1)^{-an}F_{k,a(n+1)}I\right).
\end{equation}

Thus, the proof is completed.
\end{proof}

Now, we define a $2\times 2$ matrix $R_{a}$ and then we give some new results for the $k$-Fibonacci numbers $F_{k,an}$ by matrix methods.

Define the $2\times 2$ matrix $R_{a}$ as follows:
\begin{equation}\label{e3}
R_{a}=\left[ \begin{array}{lr} L_{k,a}  & -(-1)^{a} \\ 1 &  0  \end{array}\right].
\end{equation}

By an inductive argument and using (\ref{e1}), we get
\begin{corollary}
For any integer $n\geq1$ holds:
\begin{equation}\label{e4}
R_{a}^{n}=\frac{1}{F_{k,a}}\left[ \begin{array}{lr} F_{k,a(n+1)}  & -(-1)^{a}F_{k,an} \\ F_{k,an}  &  -(-1)^{a}F_{k,a(n-1)}  \end{array}\right].
\end{equation}
\end{corollary}

Clearly the matrix $R_{a}^{n}$ satisfies the recurrence relation, for $n\geq1$
\begin{equation}\label{e5}
R_{a}^{n+1}=L_{k,a}R_{a}^{n}-(-1)^{a}R_{a}^{n-1},
\end{equation}
where $R_{a}^{0}=I$ and $R_{a}^{1}=R_{a}$.

We define $S_{a}$ be the $2\times 2$ matrix
\begin{equation}\label{e6}
S_{a}=\frac{1}{2}\left[ \begin{array}{lr} L_{k,a} & \Delta_{a} \\ 1 &  L_{k,a}  \end{array}\right],
\end{equation}
where $\Delta_{a}=L_{k,a}^{2}-4(-1)^{a}$. Then,
\begin{corollary}
For any integer $n\geq1$ holds:
\begin{equation}\label{e7}
S_{a}^{n}=\frac{1}{2F_{k,a}}\left[ \begin{array}{lr} \epsilon_{a}(n) & \Delta_{a} F_{k,an}\\ F_{k,an} & \epsilon_{a}(n) \end{array}\right].
\end{equation}
where $\epsilon_{a}(n)=2F_{k,a(n+1)}-L_{k,a}F_{k,an}$.
\end{corollary}
\begin{proof}
(By induction). For $n=1$:
\begin{equation}\label{e8}
S_{a}^{1}=\frac{1}{2}\left[ \begin{array}{lr} L_{k,a} & \Delta_{a} \\ 1 &  L_{k,a}  \end{array}\right]=\frac{1}{2F_{k,a}}\left[ \begin{array}{lr} F_{k,a}L_{k,a} & \Delta_{a} F_{k,a}\\ F_{k,a} & F_{k,a}L_{k,a} \end{array}\right]
\end{equation}
since $\epsilon_{a}(1)=F_{k,2a}=L_{k,a}F_{k,a}$. Let us suppose that the formula is true for $n-1$:
\begin{equation}\label{e9}
S_{a}^{n-1}=\frac{1}{2F_{k,a}}\left[ \begin{array}{lr} \epsilon_{a}(n-1) & \Delta_{a} F_{k,a(n-1)}\\ F_{k,a(n-1)} & \epsilon_{a}(n-1)\end{array}\right],
\end{equation}
with $\epsilon_{a}(n-1)=2F_{k,an}-L_{k,a}F_{k,a(n-1)}$.

Then,
\begin{align*}
S_{a}^{n}&=S_{a}^{n-1}S_{a}^{1}=\frac{1}{4F_{k,a}}\left[ \begin{array}{lr} \epsilon_{a}(n-1) & \Delta_{a} F_{k,a(n-1)}\\ F_{k,a(n-1)} & \epsilon_{a}(n-1) \end{array}\right]\left[ \begin{array}{lr} L_{k,a} & \Delta_{a} \\ 1 &  L_{k,a}  \end{array}\right]\\
&=\frac{1}{4F_{k,a}}\left[ \begin{array}{lr} \epsilon_{a}(n-1)L_{k,a}+\Delta_{a}F_{k,a(n-1)} & \Delta_{a}(\epsilon_{a}(n-1)+L_{k,a}F_{k,a(n-1)}) \\ \epsilon_{a}(n-1)+L_{k,a}F_{k,a(n-1)}  &  \epsilon_{a}(n-1)L_{k,a}+\Delta_{a}F_{k,a(n-1)}  \end{array}\right]\\
&=\frac{1}{2F_{k,a}}\left[ \begin{array}{lr} \epsilon_{a}(n) & \Delta_{a} F_{k,an}\\ F_{k,an} & \epsilon_{a}(n) \end{array}\right],
\end{align*}
since $2\epsilon_{a}(n)=\epsilon_{a}(n-1)L_{k,a}+\Delta_{a}F_{k,a(n-1)}$. Thus, the proof is completed.
\end{proof}

An important property of these numbers can be tested using the above result.
\begin{theorem}
For any integer $n\geq1$ holds:
\begin{equation}\label{e10}
F_{k,a(n+1)}^{2}-L_{k,a}F_{k,an}F_{k,a(n+1)}+(-1)^{a}F_{k,an}^{2}=F_{k,a}^{2}(-1)^{an}.
\end{equation}
\end{theorem}
\begin{proof}
Since $\det(S_{a})=(-1)^{a}$, $\det(S_{a}^{n})=(\det(S_{a}))^{n}=(-1)^{an}$. Moreover, since (\ref{e7}), we get $\det(S_{a}^{n})=\frac{1}{4F_{k,a}^{2}}(\epsilon_{a}(n)^{2}-\Delta_{a}F_{k,an}^{2})$. Furthermore, $$\epsilon_{a}(n)^{2}-\Delta_{a}F_{k,an}^{2}=4(F_{k,a(n+1)}^{2}-L_{k,a}F_{k,an}F_{k,a(n+1)}+(-1)^{a}F_{k,an}^{2}).$$

The proof is completed.
\end{proof}
Let us give a different proof of one of the fundamental identities of $k$-Fibonacci
and $k$-Lucas numbers, by using the matrix $S_{a}$.

\begin{theorem}
For all $n,m\in \N$,
\begin{equation}\label{e11}
F_{k,a}F_{k,a(n+m)}=F_{k,a(n+1)}F_{k,am}+F_{k,a(m+1)}F_{k,an}-L_{k,a}F_{k,an}F_{k,am}.
\end{equation}
\end{theorem}
\begin{proof}
Since $S_{a}^{n+m}=S_{a}^{n}S_{a}^{m}$, then
$$
\left[ \begin{array}{lr} \epsilon_{a}(n+m) & \Delta_{a} F_{k,a(n+m)}\\ F_{k,a(n+m)} & \epsilon_{a}(n+m) \end{array}\right]=\frac{1}{2F_{k,a}}\left[ \begin{array}{lr} \epsilon_{a}(n) & \Delta_{a} F_{k,an}\\ F_{k,an} & \epsilon_{a}(n) \end{array}\right]\left[ \begin{array}{lr} \epsilon_{a}(m) & \Delta_{a} F_{k,am}\\ F_{k,am} & \epsilon_{a}(m) \end{array}\right],
$$
where $\epsilon_{a}(n)=2F_{k,a(n+1)}-L_{k,a}F_{k,an}$. It is seen that,
\begin{equation}
2F_{k,a}S_{a}^{n+m}=\left[ \begin{array}{lr} \epsilon_{a}(n)\epsilon_{a}(m)+\Delta_{a}F_{k,an}F_{k,am} & \Delta_{a}(\epsilon_{a}(m)F_{k,an}+\epsilon_{a}(n)F_{k,am})\\ \epsilon_{a}(m)F_{k,an}+\epsilon_{a}(n)F_{k,am} & \epsilon_{a}(n)\epsilon_{a}(m)+\Delta_{a}F_{k,an}F_{k,am} \end{array}\right].
\end{equation}

Thus it follows that,
\begin{equation}
2F_{k,a}F_{k,a(n+m)}=\epsilon_{a}(m)F_{k,an}+\epsilon_{a}(n)F_{k,am},
\end{equation}
and
$$\epsilon_{a}(m)F_{k,an}+\epsilon_{a}(n)F_{k,am}=2(F_{k,a(n+1)}F_{k,am}+F_{k,a(m+1)}F_{k,an}-L_{k,a}F_{k,an}F_{k,am}).$$

Then, the proof is completed.
\end{proof}

In the particular case, if $a=1$, we obtain 
\begin{corollary}
For all $n,m\in \N$,
\begin{equation}\label{e12}
F_{k,n+m}=F_{k,m+1}F_{k,n}+F_{k,m}F_{k,n-1}.
\end{equation}
\end{corollary}

\begin{theorem}
For all $n,m\in \N$,
\begin{equation}\label{e13}
(-1)^{am}F_{k,a}F_{k,a(n-m)}=F_{k,a(m+1)}F_{k,an}-F_{k,a(n+1)}F_{k,am}.
\end{equation}
\end{theorem}
\begin{proof}
Since $S_{a}^{n-m}=S_{a}^{n}(S_{a}^{m})^{-1}$, then
$$
\left[ \begin{array}{lr} \epsilon_{a}(n-m) & \Delta_{a} F_{k,a(n-m)}\\ F_{k,a(n-m)} & \epsilon_{a}(n-m) \end{array}\right]=\frac{(-1)^{am}}{2F_{k,a}}\left[ \begin{array}{lr} \epsilon_{a}(n) & \Delta_{a} F_{k,an}\\ F_{k,an} & \epsilon_{a}(n) \end{array}\right]\left[ \begin{array}{lr} \epsilon_{a}(m) & -\Delta_{a} F_{k,am}\\ -F_{k,am} & \epsilon_{a}(m) \end{array}\right],
$$
where $\epsilon_{a}(n)=2F_{k,a(n+1)}-L_{k,a}F_{k,an}$. It is seen that,
\begin{equation}
2(-1)^{am}F_{k,a}S_{a}^{n-m}=\left[ \begin{array}{lr} \epsilon_{a}(n)\epsilon_{a}(m)-\Delta_{a}F_{k,an}F_{k,am} & \Delta_{a}(\epsilon_{a}(m)F_{k,an}-\epsilon_{a}(n)F_{k,am})\\ \epsilon_{a}(m)F_{k,an}-\epsilon_{a}(n)F_{k,am} & \epsilon_{a}(n)\epsilon_{a}(m)-\Delta_{a}F_{k,an}F_{k,am} \end{array}\right].
\end{equation}

Thus it follows that,
\begin{equation}
2(-1)^{am}F_{k,a}F_{k,a(n-m)}=\epsilon_{a}(m)F_{k,an}-\epsilon_{a}(n)F_{k,am},
\end{equation}
and
$$\epsilon_{a}(m)F_{k,an}-\epsilon_{a}(n)F_{k,am}=2(F_{k,a(m+1)}F_{k,an}-F_{k,a(n+1)}F_{k,am}).$$

Then, the proof is completed.
\end{proof}

In the particular case, if $a=1$, we obtain 
\begin{corollary}
For all $n,m\in \N$,
\begin{equation}\label{e14}
(-1)^{m}F_{k,n-m}=F_{k,m+1}F_{k,n}-F_{k,n+1}F_{k,m}.
\end{equation}
\end{corollary}

\section{Sum of $k$-Fibonacci numbers of kind $an$}
In this section, we study the sum of the $k$-Fibonacci numbers of kind $an$, with $a$ an positive integer number.
\begin{theorem}
Let $n\in \N$ and $a\in \Z$ with $a\geq 1$. Then,
\begin{equation}
\sum_{i=0}^{n}F_{k,ai}=\frac{(-1)^{a}F_{k,an}+F_{k,a}-F_{k,a(n+1)}}{1+(-1)^{a}-L_{k,a}}.
\end{equation}
\end{theorem}

\begin{proof}
It is known that $I-S_{a}^{n+1}=(I-S_{a})\sum_{i=0}^{n}S_{a}^{i}$. If $\det(I-S_{a})$ is nonzero, then we can write
\begin{equation}\label{e15}
(I-S_{a})^{-1}(I-S_{a}^{n+1})=\sum_{i=0}^{n}S_{a}^{i}=\frac{1}{2F_{k,a}}\left[ \begin{array}{lr} \sum_{i=0}^{n}\epsilon_{a}(i) & \Delta_{a} \sum_{i=0}^{n}F_{k,ai}\\ \sum_{i=0}^{n}F_{k,ai} & \sum_{i=0}^{n}\epsilon_{a}(i) \end{array}\right].
\end{equation}
where $\epsilon_{a}(i)=2F_{k,a(i+1)}-L_{k,a}F_{k,ai}$.

It is easy to see that,
\begin{equation}
\det(I-S_{a})=\left(1-\frac{1}{2}L_{k,a}\right)^{2}-\frac{1}{4}\Delta_{a}=1+(-1)^{a}-L_{k,a}
\end{equation}
is nonzero, because $a\geq 1$. If we take $\delta=1+(-1)^{a}-L_{k,a}$, then we get
\begin{equation}\label{e16}
(I-S_{a})^{-1}=\frac{1}{\delta}\left[ \begin{array}{lr} 1-\frac{1}{2}L_{k,a} & \frac{\Delta_{a}}{2} \\ \frac{1}{2} &  1-\frac{1}{2}L_{k,a}  \end{array}\right]=\frac{1}{\delta}\left[ \left(1-\frac{1}{2}L_{k,a}\right)I+\frac{1}{2}T_{a}\right],
\end{equation}
where $T_{a}=\left[ \begin{array}{lr} 0 & \Delta_{a} \\ 1 &  0  \end{array}\right]$.

Thus it is seen that,
\begin{align*}
(I-S_{a})^{-1}(I-S_{a}^{n+1})&=\frac{1}{\delta}\left[ \left(1-\frac{1}{2}L_{k,a}\right)I+\frac{1}{2}T_{a}\right](I-S_{a}^{n+1})\\
&=\frac{1}{\delta}\left[ \left(1-\frac{1}{2}L_{k,a}\right)(I-S_{a}^{n+1})+\frac{1}{2}T_{a}(I-S_{a}^{n+1})\right],
\end{align*}
where
$$T_{a}(I-S_{a}^{n+1})=\frac{1}{2F_{k,a}}\left[\begin{array}{lr}  -\Delta_{a}F_{k,a(n+1)} & \Delta_{a}(2F_{k,a}-\epsilon_{a}(n+1))\\ 2F_{k,a}-\epsilon_{a}(n+1) & -\Delta_{a}F_{k,a(n+1)} \end{array}\right].$$

Furthermore, from the identity (\ref{e15}), it follows that
\begin{align*}
\sum_{i=0}^{n}F_{k,ai}&=\frac{1}{\delta}\left(-\left(1-\frac{1}{2}L_{k,a}\right)F_{k,a(n+1)}+\frac{1}{2}(2F_{k,a}-\epsilon_{a}(n+1))\right)\\
&=\frac{1}{\delta}((-1)^{a}F_{k,an}+F_{k,a}-F_{k,a(n+1)}).
\end{align*}
\end{proof}

\begin{theorem}
Let $n\in \N$ and $a\in \Z$ with $a\geq 1$. Then,
\begin{equation}\label{e17}
\sum_{i=0}^{n}(-1)^{i}F_{k,ai}=\frac{(-1)^{a}F_{k,an}-F_{k,a}+F_{k,a(n+1)}}{1+(-1)^{a}+L_{k,a}}.
\end{equation}
\end{theorem}
\begin{proof}
We prove the theorem in two phases, by taking n as an even and odd natural number. Firstly assume that n is an even natural number. Then,
$$I+S_{a}^{n+1}=(I+S_{a})\sum_{i=0}^{n}(-1)^{i}S_{a}^{i}.$$ If $\det(I+S_{a})$ is nonzero, then we can write
\begin{equation}\label{e18}
(I+S_{a})^{-1}(I+S_{a}^{n+1})=\sum_{i=0}^{n}S_{a}^{i}=\frac{1}{2F_{k,a}}\left[ \begin{array}{lr} \sum_{i=0}^{n}(-1)^{i}\epsilon_{a}(i) & \Delta_{a} \sum_{i=0}^{n}(-1)^{i}F_{k,ai}\\ \sum_{i=0}^{n}(-1)^{i}F_{k,ai} & \sum_{i=0}^{n}(-1)^{i}\epsilon_{a}(i) \end{array}\right].
\end{equation}
where $\epsilon_{a}(i)=2F_{k,a(i+1)}-L_{k,a}F_{k,ai}$.

It is easy to see that,
\begin{equation}
\det(I+S_{a})=\left(1+\frac{1}{2}L_{k,a}\right)^{2}-\frac{1}{4}\Delta_{a}=1+(-1)^{a}+L_{k,a}
\end{equation}
is nonzero. If we take $\delta=1+(-1)^{a}+L_{k,a}$, then we get
\begin{equation}\label{e19}
(I+S_{a})^{-1}=\frac{1}{\delta}\left[ \begin{array}{lr} 1+\frac{1}{2}L_{k,a} & -\frac{\Delta_{a}}{2} \\ -\frac{1}{2} &  1+\frac{1}{2}L_{k,a}  \end{array}\right]=\frac{1}{\delta}\left[ \left(1+\frac{1}{2}L_{k,a}\right)I-\frac{1}{2}T_{a}\right],
\end{equation}
where $T_{a}=\left[ \begin{array}{lr} 0 & \Delta_{a} \\ 1 &  0  \end{array}\right]$.

Thus it is seen that,
\begin{align*}
(I+S_{a})^{-1}(I+S_{a}^{n+1})&=\frac{1}{\delta}\left[ \left(1+\frac{1}{2}L_{k,a}\right)I-\frac{1}{2}T_{a}\right](I+S_{a}^{n+1})\\
&=\frac{1}{\delta}\left[ \left(1+\frac{1}{2}L_{k,a}\right)(I+S_{a}^{n+1})-\frac{1}{2}T_{a}(I+S_{a}^{n+1})\right],
\end{align*}
where
$$T_{a}(I+S_{a}^{n+1})=\frac{1}{2F_{k,a}}\left[\begin{array}{lr}  \Delta_{a}F_{k,a(n+1)} & \Delta_{a}(2F_{k,a}+\epsilon_{a}(n+1))\\ 2F_{k,a}+\epsilon_{a}(n+1) & \Delta_{a}F_{k,a(n+1)} \end{array}\right].$$

Furthermore, from the identity (\ref{e15}), it follows that
\begin{align*}
\sum_{i=0}^{n}(-1)^{i}F_{k,ai}&=\frac{1}{\delta}\left(\left(1+\frac{1}{2}L_{k,a}\right)F_{k,a(n+1)}+\frac{1}{2}(2F_{k,a}+\epsilon_{a}(n+1))\right)\\
&=\frac{1}{\delta}((-1)^{a}F_{k,an}-F_{k,a}+F_{k,a(n+1)}).
\end{align*}

Now assume that $n$ is an odd natural number. Hence we get,
\begin{equation}\label{e20}
\sum_{i=0}^{n}(-1)^{i}F_{k,ai}=\sum_{i=0}^{n-1}(-1)^{i}F_{k,ai}-F_{k,an}
\end{equation}

Since $n$ is an odd natural number, then $n-1$ is even. Thus taking $n-1$ in (\ref{e17}) and using it in (\ref{e20}), the proof is completed.
\end{proof}

\medskip

\begin{thebibliography}{99}

\bibitem[1]{1} 
S. Vajda, Fibonacci and Lucas Numbers and the Golden Section, Theory and Applications, Ellis Horwood Limited, 1989.

\bibitem[2]{2}
S. Falcon, A. Plaza, On the Fibonacci $k$-numbers, Chaos, Soliton Fract. 32 (5) (2007) 1615-1624.

\bibitem[3]{3}
S. Falcon, A. Plaza, The $k$-Fibonacci sequence and the Pascal 2-triangle, Chaos, Soliton Fract. 33 (1) (2007) 38-49.

\bibitem[4]{4}
S. Falcon, A. Plaza, The $k$-Fibonacci hyperbolic functions, Chaos, Soliton Fract. 38 (2) (2008) 409-420.

\bibitem[5]{5}
B. Demirturk, Fibonacci and Lucas Sums by Matrix Methods, International Mathematical Forum, 5, 2010, no. 3, 99-107.

\bibitem[6]{6}
Cerda-Morales, G. On generalized Fibonacci and Lucas numbers by matrix methods, Hacettepe Journal of Mathematics and Statistics, Vol. 42 (2), pp. 173-179 (2013).
\end{thebibliography}
\end{document}